\documentclass[11pt]{article}
\usepackage[pagebackref,colorlinks=true,urlcolor=blue,linkcolor=red,citecolor=red]{hyperref}

\usepackage{epsfig, graphics, graphicx}
\usepackage{subcaption, comment, bm}
\usepackage[percent]{overpic}
\usepackage{float}
\usepackage{amssymb,bm}
\usepackage{color}
\usepackage{amsmath}
\usepackage{amsthm}
\usepackage[]{amsfonts}
\usepackage{fancyhdr}
\usepackage[]{graphicx,wrapfig}
\usepackage{enumitem}
\usepackage{array}
\usepackage{makecell}
\usepackage{graphicx}
\usepackage[dvipsnames]{xcolor}

\newcommand{\Lc}{\mathcal{L}}

\newcommand{\Tc}{\mathcal{T}}

\newcommand{\Vc}{\mathcal{V}}

\newcommand{\Xc}{\mathcal{X}}

\newcommand{\Db}{\mathbb{D}}

\newcommand{\Rb}{\mathbb{R}}

\newcommand{\Zb}{\mathbb{Z}}

\newcommand{\vf}{\textbf{\textit{f}}}

\newcommand{\vx}{\textbf{\textit{x}}}

\newcommand{\vw}{\textbf{\textit{w}}}
\newcommand{\vu}{\textbf{\textit{u}}}
\newcommand{\vv}{\textbf{\textit{v}}}

\usepackage{amsthm}
\usepackage{amsfonts}
\newcommand{\Beq}{\begin{equation}}
\newcommand{\Eeq}{\end{equation}}
\newcommand{\beq}{\begin{equation*}}
\newcommand{\eeq}{\end{equation*}}
\newcommand{\bal}{\begin{align}}
\newcommand{\eal}{\end{align}}

\graphicspath{{./Figs/}}

\newtheorem*{theor}{Theorem}
\newtheorem{thr}{Theorem}
\newtheorem{defn}{Definition}
\newtheorem{rem}{Remark}
\newtheorem{cor}{Corollary}
\newtheorem{lem}{Lemma}
\theoremstyle{definition}

\newtheorem{hypo}{Hypothesis}

\title{\vspace{-1cm} Reconstruction of scalar functions and vector fields\\ from weighted V-line transforms with swinging branches}
\author{Gaik Ambartsoumian\thanks{Department of Mathematics, University of Texas at Arlington, Arlington, TX, United States of America. \url{gambarts@uta.edu}}\and Rohit Kumar Mishra\thanks{Department of Mathematics, Indian Institute of Technology, Gandhinagar, Gujarat, India. \url{rohit.m@iitgn.ac.in}, \url{rohittifr2011@gmail.com}}
\and Indrani Zamindar\thanks{Department of Mathematics, Indian Institute of Technology, Gandhinagar, Gujarat, India. \url{indranizamindar@iitgn.ac.in}}}
\begin{document}
\date{}
\maketitle
\begin{abstract}
Weighted V-line transforms map a symmetric tensor field of order $m\ge0$ to a linear combination of certain integrals of those fields along two rays emanating from the same vertex. A significant focus of current research in integral geometry centers on the inversion of V-line transforms in formally determined setups. Of particular interest are the restrictions of these operators in which the vertices of integration trajectories can be anywhere inside the support of the field, while the directions of the pair of rays, often called branches of the V-line, are determined by the vertex location.
Such transforms have been thoroughly investigated when the branch directions are either constant or radial. In addition to that, in most of the prior research on this subject,  it was assumed that the weights of integration along each branch are the same. In this paper we analyze the transforms defined on scalar functions and vector fields, satisfying a much weaker assumption on the branch directions. The weights restriction is also lifted in all but one setup. Consequently, we extend multiple previously known results on the kernel description, injectivity, and inversion of the transforms with simplifying assumptions and prove pertinent statements for more general setups not studied before.
\end{abstract}
\section{Introduction}
During the last few decades, mathematicians have thoroughly investigated numerous inverse problems related to the recovery of tensor fields from various generalized X-ray transforms integrating along straight lines in $\Rb^n$ or along geodesics on a Riemannian manifold. The operators of interest include the longitudinal, transverse, mixed and momentum ray transforms of a tensor field, and the main efforts have focused on the study of injectivity, inversion formulas, null space, range characterizations, stability estimates, and support theorems of these operators (e.g. see \cite{abhishek2020support, abhishek2019support,   Boman_Sharafutdinov_Stability, Holman2013,  Francois_Jonnas_Review_2019, krishnan2019momentum,krishnan2020momentum,krishnan2019solenoidal, krishnan2009support, Rohit_Suman_2021,  Paternain_Salo_Uhlmann_review_2014, Paternain_Salo_Uhlmann_Book, Kamran_Otmar_Xray_2tensor, Kamran_Tamasan_Range, 
Sharafutdinov_Michal_Uhlmann_Regularity, Sharafutdinov_Book} and the references there). In higher dimensions ($n\ge 3$), the inversion problems for both longitudinal and transverse ray transforms are overdetermined. In such setups, researchers have studied formally determined problems of recovering the unknown field from restricted data sets satisfying miscellaneous conditions (e.g. see \cite{Denisjuk_Paper, krishnan2018microlocal, mishra2020full, Rohit_Chandni_transverse, Vertgeim2000}).

More recently, a new direction of generalizing ray transforms has become the focus of scientific inquiry. Motivated by imaging applications utilizing scattered particles, these transforms integrate along trajectories that contain ``a vertex'' or ``a corner'', which corresponds to the scattering location \cite{amb-book}. The operators of interest here include the divergent beam transform (DBT) mapping a symmetric tensor field of order $m\ge0$ to its integrals along half-lines \cite{Venky_Divergent_Beam_2025, kuchment2017inversion}, the V-line transforms (VLT) (also known as broken ray transforms (BRT)) and the star transforms defined, respectively, as linear combinations of a pair or more DBTs with a common vertex (e.g. see \cite{star-dual, amb-lat_2019, Amb_Lat_star, Gaik_Mohammad_Rohit, Gaik_Lat_Rohit_numerics, Gaik_Indrani_Rohit, Gaik_Indrani_Rohit_numerics, Ambartsoumian_Moon_broken_ray_article, ZSM-star-14} and the references therein), as well as various conical Radon transforms \cite{amb-lat_2019, Gouia_Amb_V-line, Palamodov2017}. A slightly different operator, also called a broken ray transform, integrates tensor fields along broken rays that reflect (possibly multiple times) from the boundary of one or more obstacles \cite{Ilmavirta_Parternain, Ilmavirta_Salo_2016,  Shubham_Manas_Jesse}. The rigorous definitions of the operators relevant to this article are presented in Section \ref{sec:def}.

Since the set of all divergent beams (i.e. half lines) in $\mathbb{R}^n$ has $2n-1$ dimensions, the inversion of each of these transforms from a complete data set is an overdetermined problem for any $n\ge 2$. Therefore, most of the research on this subject has concentrated on the study of the restricted versions of such operators with $n$ degrees of freedom. The latter can be informally split into two categories: transforms in which the vertices of integration trajectories are restricted to a hypersurface located \textit{outside} or on the boundary of the support of the field (e.g. see \cite{Moon_Haltmeir_Daniela, Anamika_Indrani_Rohit, Moon_Haltmeir_CRT_2017} and references therein), and transforms in which those vertices can be anywhere \textit{inside} the support of the field (e.g. see \cite{amb-lat_2019, Amb_Lat_star, Gaik_Mohammad_Rohit, Gaik_Lat_Rohit_numerics, Gaik_Indrani_Rohit, Gaik_Indrani_Rohit_numerics, Ambartsoumian_Moon_broken_ray_article, Gouia_Amb_V-line, Venky_Divergent_Beam_2025, kuchment2017inversion, Palamodov2017, ZSM-star-14}). The operators from the first group appear in the mathematical models of Compton cameras, while those from the second group play a prominent role in single scattering optical and X-ray tomographies \cite{amb-book}. The mathematical apparatus used to analyze these operators is also different for each category. For example, many problems related to the first group can be modified into the equivalent problems about (classical) ray transforms by continuing the data to the missing half-lines with appropriate symmetry conditions. Clearly, such tricks will not work for the transforms from the second group, making their study a more challenging endeavor. Our article deals with a large class of operators from the second group, generalizing the results of a host of previous works.

The transforms studied in this paper map a scalar or a vector field $f$ defined in $\mathbb{R}^2$  to its weighted integrals along various $2$-dimensional families of V-lines\footnote{A divergent beam transform can be expressed as a weighted V-line transform, where the weight along one of the branches is chosen to be zero.} with the following common features. Each point of the support of $f$ is a vertex of exactly one V-line of the given family. In other words, each vertex location $\vx$ uniquely identifies the directions $\vu(\vx)$ and $\vv(\vx)$ of the branches 
of the V-line emanating from that vertex. Therefore, the VLT can be parametrized by the coordinates $\vx$ of the vertices of its integration trajectories. Some prominent examples of such setups include translation invariant VLTs\footnote{In imaging applications this corresponds to linear arrays of photon emitters and detectors, each collimated in a single direction, i.e. $\vu(\vx)$ and $\vv(\vx)$ are constant. It is often said that the linear arrays have a focal point at infinity, while the circular arc detectors have a finite focal point at the center of the circle.} \cite{amb-lat_2019, Amb_Lat_star, Gaik_Mohammad_Rohit, Gaik_Lat_Rohit_numerics, Gaik_Indrani_Rohit, Gaik_Indrani_Rohit_numerics, FMS-PhysRev-10, FMS-09, gouia2014analytical, Gouia_Amb_V-line, Kats_Krylov-15}, rotation invariant VLTs \cite{amb_2012, Ambartsoumian_Moon_broken_ray_article,ambartsoumian2016numerical,  Rahul_tensor, Rahul_Rohit_Manmohan_vector, Sherson}, and VLTs appearing in imaging modalities using circular (arc) arrays of emitters and receivers (see Figure \ref{fig:Detector}) \cite{amb-book, Kats_Krylov-13}. We also assume that the integral curves of the vector fields $\vu(\vx)$ and $\vv(\vx)$ coincide with straight line segments in the image domain. For V-line branches representing the incidence field of radiation (e.g. corresponding to $\vu(\vx)$) the latter condition is necessary (see \cite{amb-book} for more details). For the V-line branches representing the scattered beam (corresponding to $\vv(\vx)$), that requirement is satisfied in all setups with linear and circular arrays of detectors discussed above, as well as in many other cases not considered before.
Given these fairly general assumptions, the methodology presented in this paper enables us to extend various results on the kernel description, injectivity, and inversion of the appropriate transforms obtained in \cite{amb-lat_2019, Gaik_Mohammad_Rohit, Venky_Divergent_Beam_2025,  Kats_Krylov-13, Sherson}, and prove pertinent statements for more general setups not studied before (see Section \ref{sec:main-results} for a detailed list).

\vspace{4mm}



\begin{figure}[h!]
    \centering
    \includegraphics[width=.8\linewidth]{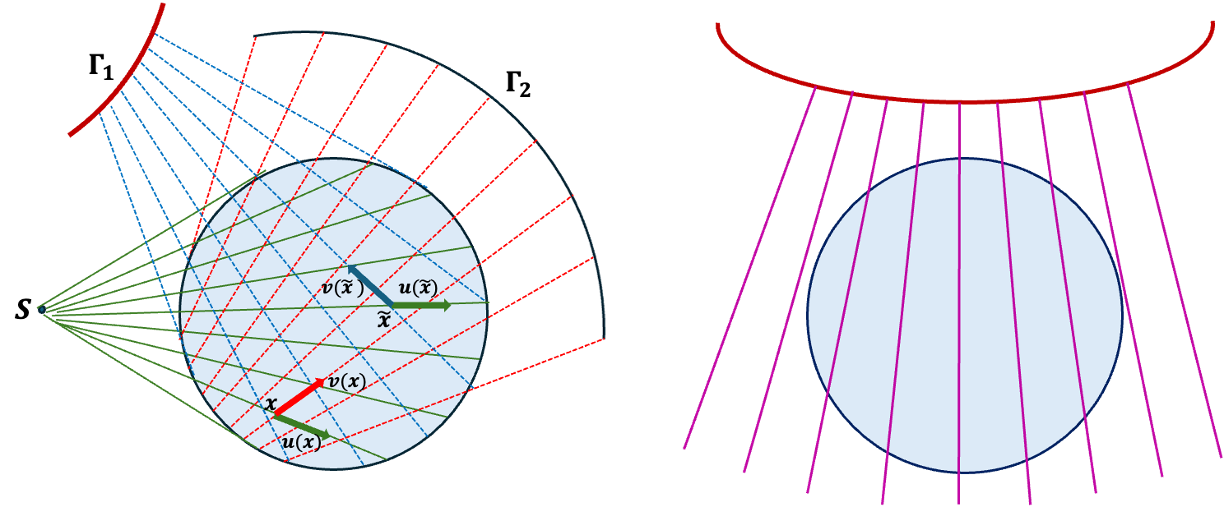}
    \caption{The sketch on the left describes a simple setup of single scattering tomography. The source $S$ emits radiation along certain rays. The single scattered photons are then captured by either a convex ($\Gamma_1$) or a concave ($\Gamma_2$) array of collimated detectors.  Concave-type detectors are often used in CT, while convex detectors are used in pin-hole cameras in nuclear imaging. Under certain assumptions, the knowledge of intensity of incoming and scattered radiation for each source-detector pair provides the VLT of the attenuation coefficient $\mu_t(\vx)$ of the medium (e.g. see \cite{amb-book}).\\
    It is easy to see that the integral curves of the corresponding vector fields $\vu(\vx)$ and $\vv(\vx)$ here are straight line segments. In particular, if the detector arrays are placed along circular arcs, then the resulting vector fields are focal. That is, for each detector $\Gamma_i$ there is a fixed point $\vx_0^i$ (the focus of $\Gamma_i$), such that the rays detected by the detector $\Gamma_i$ pass through $\vx_0^i$. In other words, for all $\vx$ inside the image domain, one can define the vector fields as $\displaystyle \vu(\vx) =  \frac{\vx_0^1 -  \vx}{|\vx_0^1 -  \vx|}$ and $\displaystyle \vv(\vx) =  \frac{\vx_0^2 -  \vx}{|\vx_0^2 -  \vx|}$.
    In the case when $\Gamma_1$ are $\Gamma_2$ are flat (i.e. when the foci of the detector arrays are at infinity), the vector fields $\vu$ and $\vv$ are constant.}    
\label{fig:Detector}
\end{figure}

\begin{rem}
    Almost all of the previously studied setups of single scattering tomography use either circular or linear arrays of detectors. In this work we substantially relax those restrictions, by only assuming that the integral curves of the vector fields $\vu(\vx)$ and $\vv(\vx)$ are straight line segments inside the image domain. Therefore, one can use arbitrary convex or concave detectors. The resulting integral curves of the vector fields of branch directions are \textit{rays that swing} along the surface of the detector  (see the sketch on the right in Figure \ref{fig:Detector}).    
\end{rem}


The rest of the paper is organized as follows. In Section \ref{sec:def} we introduce the notations, definitions, and assumptions about the transforms studied in the article. Section \ref{sec:main-results} enumerates the main results of this work in the form of an itemized list and two tables, which should help the reader navigate through the paper and quickly locate the desired theorems. Section \ref{sec: a = 0} describes the statements about the divergent beam transform and its moments acting on scalar functions, as well as the longitudinal and transverse divergent beam transforms defined on vector fields. In Section \ref{sec: a not 0 scalar} we consider the weighted V-line transform of a scalar function $h$ and present a method for its inversion. Section \ref{sec: a = 1} delineates the reconstruction of a vector field from its longitudinal/transverse V-line transforms and their first moments. Section \ref{sec: a not 0 vector} discusses the reconstruction of a vector field from various combinations of its weighted V-line transforms with constant branch directions $\vu$, $\vv$. We finish the paper with a description of open questions and future work in Section \ref{sec:remarks}, and acknowledgments in Section \ref{sec:acknowledge}. 


\vspace{20mm}

\section{Preliminaries}\label{sec:def}
\noindent Throughout the paper, we use bold font letters to denote vectors and regular font letters to denote scalars. Let $\Omega\subseteq\mathbb{R}^2$ be an open set, and $S^1(\Omega)$ be the space of vector fields defined on $\Omega$. We denote by $C_c^2(S^1; \Omega)$ the space of twice continuously differentiable, compactly supported vector fields on $\Omega$. Since we are interested in recovering scalar functions and vector fields of compact support, with appropriate scaling one can consider only those that are supported in the unit disc.
\vspace{2mm}

In all statements of this article we assume that the following conditions hold.
\begin{hypo}\label{Hypo} Let $\Db$ be the unit disc.
    \begin{itemize}
     \item At each $\vx\in \Db$, the vectors $\vu(\vx)$ and $\vv(\vx)$ are linearly independent and have unit length.
    \item The vector fields $\vu,\vv\in C^1(S^1;\Db)$, and their integral curves are straight line segments in $\Db$.
\end{itemize}
\end{hypo}

Below, we define a set of integral transforms for scalar functions and vector fields in $\Rb^2$, which are the primary objects of our study.  We start with introducing a weighted V-line transform of a scalar function $h$.


\begin{defn}\label{def:weighted V-line transform}
Let $\alpha\in\mathbb{R}$ 
and $h\in C_c^1(\Db)$. The (weighted) \textbf{V-line transform} of $h$ is defined as
\begin{align}\label{weighted_V-line}
\Vc_{\alpha}h(\vx) = \int_0^{\infty} h(\vx+t\vu(\vx))\,dt +\alpha \int_0^{\infty} h(\vx+t\vv(\vx))\, dt, \;\;\vx\in\Db.
\end{align} 
\end{defn}
\noindent This transform depends on the choice of vector fields $\vu$ and $\vv$, but we do not include them in the notation $\Vc_{\alpha}$ because $\vu$ and $\vv$ are assumed to be known, and it will be clear from the discussion what $\vu$, $\vv$ are in each particular case.
For $\alpha=0$, this transform reduces to the \textit{divergent beam transform}, which we will occasionally denote by $\Xc_\vu$ to emphasize the direction, that is, $\Xc_\vu h\equiv\Vc_0 h$.

\begin{defn}\label{def:focal divergent moment beam}
 Let $h\in C_c^1(\Db)$ and $k\ge 0$ be an integer. The \textbf{$k^{th}$ moment divergent beam transform} of $h$ is defined as
\begin{align}\label{focal_divergent_beam}
\Vc^k_{0}h(\vx) = \int_0^{\infty} t^kh(\vx+t\vu(\vx))\,dt, \;\;\vx\in\Db.
\end{align}   
\end{defn}

Next, we introduce a set of related integral transforms acting on a vector field $\vf$ in $\Rb^2$. As in the previous definition, we take the weight $\alpha$ to be a fixed real number, and $\vu$, $\vv$ be known vector fields, whose integral curves are straight line segments in $\Db$.


\begin{defn}
    Let $\vf \in C_c^1(S^1;\Db)$ and $k\ge 0$ be an integer. The \textbf{ $k^{th}$ moment longitudinal V-line transform} of $\vf$ is defined as
\begin{align}\label{mom_longi_def}
        \Lc_\alpha^k\vf(\vx)= -\int_0^\infty t^k \vu(\vx) \cdot \vf(\vx+t\vu(\vx))\, dt + \alpha\int_0^\infty t^k \vv(\vx) \cdot \vf(\vx+t\vv(\vx))\, dt, \;\;\vx\in\Db.
    \end{align}
\end{defn}

\begin{defn}
    Let $\vf \in C_c^1(S^1;\Db)$ and $k\ge 0$ be an integer. The \textbf{$k^{th}$ moment transverse V-line transform} of $\vf$ is defined as
\begin{align}\label{mom_trans_def}
        \Tc_\alpha^k\vf(\vx)= -\int_0^\infty t^k \vu^\perp(\vx) \cdot \vf(\vx+t\vu(\vx))\, dt + \alpha \int_0^\infty t^k \vv^\perp(\vx) \cdot \vf(\vx+t\vv(\vx))\, dt, \;\;\vx\in\Db,
    \end{align}
where $\vu^\perp (\vx)=(u_1(\vx), u_2(\vx))^\perp= (-u_2(\vx), u_1(\vx)).$    
\end{defn}
\begin{rem}
    When $k=0$, we call these transforms the longitudinal VLT and the transverse VLT, and denote them by $\Lc_\alpha\vf$ and $\Tc_\alpha\vf$.
\end{rem}

\noindent It is easy to observe by a simple calculation that $\displaystyle \Lc_\alpha\vf^{\perp} =- \Tc_\alpha\vf$ and $\Lc_\alpha^1\vf^{\perp} =- \Tc_\alpha^1\vf$. As 
in the definition of the (weighted) V-line transform, 
we do not include vector fields $\vu$ and $\vv$ in the notation of these transforms, since $\vu$ and $\vv$ are assumed to be known vector fields satisfying Hypothesis \ref{Hypo}.

\vspace{0.5cm}

\noindent 
For a scalar function $h$  and a vector field $\vf =(f_1,f_2)$, we use the following notations 
\begin{align}\label{eq: definition of div and curl}
\nabla h = \left(\frac{\partial h}{\partial x_1}, \frac{\partial h}{\partial x_2}\right), \ \  \nabla^\perp h = \left(-\frac{\partial h}{\partial x_2}, \frac{\partial h}{\partial x_1}\right), \ \    \delta \vf =  \frac{\partial f_1}{\partial x_1}+ \frac{\partial f_2}{\partial x_2},\ \  
\delta^\perp \vf =  \frac{\partial f_2}{\partial x_1}- \frac{\partial f_1}{\partial x_2}.
\end{align}
For given vector fields $\vu(\vx),\vv(\vx)$ satisfying Hypothesis \ref{Hypo}, we define 
\begin{align}\label{eq:definition c_uv}
    c_{uv}(\vx):=\vu(\vx)\cdot \vv(\vx) \quad \mbox{and} \quad 
    c^{\perp}_{uv}(\vx):=\vu(\vx)\cdot \vv^\perp(\vx), \quad    \mbox{ where } \vv^\perp(\vx)=(-v_2(\vx),v_1(\vx)).
\end{align}
It is easy to verify that $\displaystyle c_{vu}(\vx)=c_{uv}(\vx)$ and $c^\perp_{vu}(\vx)=-c^{\perp}_{uv}(\vx)$. Let $D_\vu=\vu \cdot \nabla$, and $D^\perp_\vu= \vu^\perp \cdot \nabla$ denote the directional derivatives in the directions $\vu$ and $\vu^\perp$, respectively. One can check by a direct calculation that
\begin{align}
 D_\vu = c_{uv}D_\vv &+ c^\perp_{uv}D^\perp_\vv \quad &\mbox{and} \qquad \quad \ \  D_\vv = c_{vu}D_\vu &+c^\perp_{vu}D^\perp_\vu\label{decomp: D_u and D_v}\\
  D_\vu D^\perp_\vu -  D^\perp_\vu D_\vu &=- \left(\delta \vu\right)D^\perp_\vu \quad  &\mbox{and} \quad D_\vv D^\perp_\vv -  D^\perp_\vv D_\vv = &- \left(\delta \vv\right) D^\perp_\vv.\label{eq:D_uD^perp_u and D_vD^perp_v}
\end{align}
Since the vector fields $\vu(\vx)$, $\vv(\vx)$ are of unit length for every $\vx \in \Db$, and their integral curves are straight line segments in $\Db$, we have the following useful relations for the operators defined above.  
\begin{align}\label{D_u(u)=0}
 D_\vu \vu(\vx) =0,  \quad \quad
     D_\vu \vu^\perp(\vx) =0, \quad \mbox{and} \quad  D^\perp_\vu \vu(\vx) = \left(\delta \vu\right)(\vx)\;\vu^\perp(\vx).\\
       D_\vv \vv(\vx) =0, \quad \quad
       D_\vv \vv^\perp(\vx) =0, \quad \mbox{and } \quad D^\perp_\vv \vv (\vx)=  \left(\delta \vv\right)(\vx)\;\vv^\perp(\vx).
\end{align}
We also have 
\begin{align}
    D_\vv c_{uv}&= \left(\delta \vu\right)(c^\perp_{vu})^2,\quad \qquad \qquad  D_\vv c^\perp_{uv} = \left(\delta \vu\right) c^\perp_{vu}c_{uv},\label{eq:D_v(c_uv)}\\
   D_\vu \vv &= c^\perp_{uv}\left(\delta \vv\right) \vv^\perp,\quad \qquad \qquad D_\vu \vv^\perp = -c^\perp_{uv}\left(\delta \vv\right) \vv,\label{eq:D_u(v)}\\
    D_\vv \vu &= c^\perp_{vu}\left(\delta \vu\right) \vu^\perp,\quad \qquad \qquad D_\vv \vu^\perp = -c^\perp_{vu}\left(\delta \vu\right)\vu\label{eq:D_v(u)}. 
\end{align}
Finally, in our derivations we will need the following well-known decomposition result.
\begin{theor}[\cite{derevtsov3}]
For any $\vf \in C^2_c(S^1;\Db)$, there exist unique smooth functions $\varphi$ and $\psi$ such that 
\begin{align}\label{eq:decomposition}
    \vf= \nabla\varphi+ \nabla^{\perp}\psi, \quad \varphi|_{\partial \Db}=0, \quad \psi|_{\partial \Db}=0. 
\end{align}
\end{theor}


\section{Listing of the main results}\label{sec:main-results}

The primary goal of this article is to analyze the injectivity and inversion of the transforms defined in Section \ref{sec:def}. In this section, we briefly describe our main results and compare them with previously known statements in more restrictive setups. We break the discussion into the following parts: 
\begin{itemize}
    \item $\alpha = 0$ (Section \ref{sec: a = 0}). This case corresponds to DBTs of scalar functions and vector fields. The inversion of the $k$-th moment DBT of scalar functions is rather straightforward. The vector field setup is more involved, due to the presence of non-trivial kernels. In \cite{Venky_Divergent_Beam_2025}, the authors have derived a method for recovery of symmetric $m$-tensor fields in $\Rb^n$ from a set of $k$-th moments DBTs, when the vector field $\vu$ is constant. Here, we work with a variable vector field $\vu$, whose integral curves are straight line segments. In the case of a constant $\vu$ in $\mathbb{R}^2$, our reconstruction results for vector fields differ from those obtained in \cite{Venky_Divergent_Beam_2025}, as we use longitudinal and transverse DBTs. For more details, see the discussion in Section \ref{sec: a = 0}.
    

    \item $\alpha \neq  0$ for \textit{scalar functions} with \textit{variable $\vu,\vv$} (Section \ref{sec: a not 0 scalar}). When  $\alpha$ tends to $0$, the results here coincide with those discussed in Section \ref{sec: a = 0}. The inversion in this general case is more complicated than in the previously known special cases. When vector fields $\vu$, $\vv$ are constant, this case corresponds to the weighted V-line transforms of scalar functions discussed in \cite{amb-lat_2019}.

     \item $\alpha = 1$ for \textit{vector fields} with \textit{variable $\vu,\vv$} (Section \ref{sec: a = 1}). In the case of constant $\vu,\vv$, our recovery of 
     vector fields coincides with the formulas derived in \cite{Gaik_Mohammad_Rohit}. 

    \item $\alpha\neq 0, 1$ for \textit{vector fields} with \textit{constant} $\vu$, $\vv$ (Section \ref{sec: a not 0 vector}). 
    When $\alpha$ tends to $0$ or $1$, the results obtained here coincide with the results listed in the first two bullet points above for constant vector fields $\vu$ and $\vv$. An inversion of a related operator (the star transform for vector fields) is derived in \cite{Gaik_Mohammad_Rohit}, but the approach there is different from the one obtained in this paper.    
\end{itemize}

\noindent The following three tables are intended to help the reader navigate through the text of the article. 

\vspace{1mm}

\begin{center}
\begin{tabular}{|p{5.5cm}|p{4.5cm}|p{5.97cm}|}
\hline
\centering $\alpha = 0$ & 
\centering $\alpha = 1$ & 
\centering $\alpha \neq 0,1$ \tabularnewline
\hline

\centering $k^{th}$ moment divergent beam transform for scalar fields, where $k \in \{0\}\bigcup\mathbb{N}$ \\
($\Vc_0^k$) & 
\centering V-line transform for scalar fields \\
($\Vc_1$) & 
\centering Weighted V-line transform for scalar fields \\
($\Vc_\alpha$) \tabularnewline
\hline

\centering Longitudinal and transverse divergent beam transforms for vector fields \\
($\Lc_0$, $\Tc_0$) &
\centering Longitudinal and transverse V-line transforms and their $1^{st}$ moments for vector fields \\
($\Lc_1$, $\Tc_1$, $\Lc_1^1$, $\Tc_1^1$) &
\centering Longitudinal and transverse weighted V-line transforms and their $1^{st}$ moments for vector fields with constant $\vu$, $\vv$ \\
($\Lc_\alpha$, $\Tc_\alpha$, $\Lc_\alpha^1$, $\Tc_\alpha^1$) \tabularnewline
\hline
\end{tabular}
\captionof{table}{The list of integral transforms considered in this work.}
\end{center}


\vspace{3mm}

\begin{center}
\begin{minipage}[t]{0.65\linewidth}
\centering
\begin{tabular}{|c|c|c|c|c|c|c|c|c|c|}
\hline
Recovery & \multicolumn{2}{c|}{$h$} & \multicolumn{7}{c|}{$\vf$} \\
\hline
Data set & 
$\Vc_0^k h$ & 
$\Vc_{\alpha} h$ & 
\shortstack{$\Lc_0 \vf$ \\ $\Tc_0 \vf$} & 
\shortstack{$\Lc_1 \vf$ \\ $\Tc_1 \vf$} & 
\shortstack{$\Lc_1 \vf$ \\ $\Lc_1^1 \vf$} & 
\shortstack{$\Tc_1 \vf$ \\ $\Tc_1^1 \vf$} & 
\shortstack{$\Lc_{\alpha} \vf$ \\ $\Tc_{\alpha} \vf$} & 
\shortstack{$\Lc_{\alpha} \vf$ \\ $\Lc_{\alpha}^1 \vf$} & 
\shortstack{$\Tc_{\alpha} \vf$ \\ $\Tc_{\alpha}^1 \vf$} \\
\hline
Theorem & 
\ref{Div_beam_focal} & 
\ref{scalar_focal_rec} & 
\ref{th:inversion of L and T DDBT} & 
\ref{th:Inversion of L and T V-line} & 
\ref{th:Inversion of L and L1 V-line} & 
\ref{th:Inversion of T and T1 V-line} & 
\ref{th:f from La and Ta} & 
\ref{Vector_rec from La and La1} & 
\ref{Vector_rec from Ta and Ta1} \\
\hline
\end{tabular}
\end{minipage}
\hspace{1.55em} 
\begin{minipage}[t]{0.3\linewidth}
\centering
\begin{tabular}{|c|c|c|}
\hline
\multicolumn{3}{|c|}{Kernel Description} \\ 
\hline
\shortstack{$\Lc_0 \vf$ \\ $\Tc_0 \vf$}  &\shortstack{$\Lc_1 \vf$ \\ $\Tc_1 \vf$}  & \shortstack{$\Lc_{\alpha} \vf$ \\ $\Tc_{\alpha} \vf$} \\
\hline
\ref{th:kernel of L and T DDBT}& \ref{th:kernel of L and T V-line} & \ref{th:kernel of L_alpf} \\
\hline
\end{tabular}
\end{minipage}
\captionof{table}{The list of theorems about the recovery of unknown scalar or vector fields from various combinations of those operators and theorems about the kernel descriptions.}
\end{center}

\section{Divergent beam transforms (\texorpdfstring{$\alpha = 0$}{})}\label{sec: a = 0} 
When $\alpha=0$, the weighted VLT has no contribution from the integral along $\vv$, and the transform integrates only along one ray (in the direction $\vu$) starting from the vertex $\vx$. 
The resulting operator coincides with the divergent beam transform. This section is devoted to the study of injectivity and invertibility of the DBT and its moments for scalar functions and vector fields.

A scalar function $h$ can be recovered from any $k$-th moment DBT with a variable vector field of directions $\vu$. Namely,
\begin{thr}\label{Div_beam_focal}
    For any fixed $\displaystyle k \in \Zb_+\cup\{0\}$, $\Vc_0^k h$ determines $h\in C_c^1(\Db)$ uniquely and explicitly.
\end{thr}
\begin{proof}
    The following relations hold by the fundamental theorem of calculus: 
    \begin{align*}
        D_\vu \Vc_0^0 h(\vx) =  -  h(\vx),\;\;
        D_\vu \Vc_0^k h(\vx) =  - k \Vc_0^{k-1} h(\vx), \;\;k\ge1.
    \end{align*}
   Applying $D_\vu$ repeatedly $k+1$ times produces the expression
     \begin{align}
       h(\vx) =  \frac{(-1)^{k+1}}{k!} D^{k+1}_\vu \Vc_0^k h(\vx),
    \end{align}
which concludes the claim. 
\end{proof}

In a recent work \cite{Venky_Divergent_Beam_2025}, the authors showed for $m\ge1$ the possibility of recovering a symmetric $m$-tensor field in $\Rb^n$ from a set of $k$-th moment longitudinal DBTs corresponding to different constant vector fields $\vu$. For $m=1$, their work presents a componentwise reconstruction of a vector field from a single moment along two different fixed directions. Their approach is based on the recovery of the projection of the unknown vector field along the direction of integration from the given data. Hence, in $\mathbb{R}^2$, one can recover the unknown vector field from its $k$-th moment longitudinal DBT along two linearly independent directions. In our case, $\vu$ does not have to be constant. Instead, we consider a weaker condition, that the integral curves of $\vu$ are straight line segments. But even if we choose $\vu$ to be constant, our results for vector fields are different from those in \cite{Venky_Divergent_Beam_2025}, since we use two different (longitudinal and transverse) DBTs in a single fixed direction.

Next, we discuss the injectivity and inversion of the longitudinal and transverse DBT defined on vector fields in $\mathbb{R}^2$. We show that the kernels of these integral transforms are non-trivial, and an unknown vector field can be recovered from a combination of those transforms. In some special cases (e.g. for conservative or solenoidal vector fields), only one of the transforms is needed.
\begin{thr}[Kernel Description]\label{th:kernel of L and T DDBT}
    Let $\vf\in C^1_c(S^1;\Db)$. Then, we have 
    \begin{itemize}
        \item[(i)] $\Lc_0\vf=0$ if and only if $\displaystyle \vf = \varphi\vu^\perp$, for some $\varphi \in C_c^1(\Db)$.
        \item[(ii)]$\Tc_0\vf=0$ if and only if $\displaystyle \vf = \varphi\vu$, for some $\varphi \in C_c^1(\Db)$.
    \end{itemize}
\end{thr}
\begin{proof}
    \begin{itemize}
        \item [\textit{(i)}] Suppose $\Lc_0\vf=0$. Then differentiating $\Lc_0\vf$ in the direction of $\vu$ we get $D_{\vu}\Lc_0\vf=0$, which implies $\langle\vf,\vu \rangle=0.$ Hence we have $\vf=\varphi \vu^\perp$ for some scalar function $\varphi \in C_c^1(\Db).$\\
    Conversely, let $\vf=\varphi \vu^\perp$, then 
    \begin{align*}
        \Lc_0\vf(\vx) &= -\int_0^{\infty} \vu(\vx) \cdot\{\varphi(\vx+t\vu(\vx)) \vu^\perp(\vx+t\vu(\vx)) \}\,dt\\
        &=  -\int_0^{\infty} \varphi(\vx+t\vu(\vx))
        \{ \vu(\vx) \cdot\vu^\perp(\vx+t\vu(\vx)) \}\,dt
    \end{align*}
    Since integral curves of the vector field $\vu(\vx)$ are straight lines, we have $\vu(\vx+t\vu(\vx))=\vu(\vx)$, which implies  $\vu^\perp(\vx+t\vu(\vx))=\vu^\perp(\vx)$. Using this property, we get $\Lc_0\vf(\vx)=0$.

    \item[\textit{(ii)}] Suppose $\Tc_0\vf=0$. Then we have $D_{\vu}\Tc_0\vf=0$, which implies $\langle\vf,\vu^\perp \rangle=0.$ Thus we get $\vf=\varphi \vu$ for some scalar function $\varphi \in C_c^1(\Db).$\\
Next, if $\vf=\varphi \vu$, then 
\begin{align*}
        \Tc_0\vf(\vx)  &=  -\int_0^{\infty} \varphi(\vx+t\vu(\vx))
        \{ \vu^\perp(\vx) \cdot\vu(\vx+t\vu(\vx)) \}\,dt\\
        &= -\int_0^{\infty} \varphi(\vx+t\vu(\vx))
        \{ \vu^\perp(\vx) \cdot\vu(\vx) \}\,dt=0. \qedhere
    \end{align*}    
 \end{itemize}    
\end{proof}

\begin{thr}[Inversion]\label{th:inversion of L and T DDBT}
Let $\vf\in C^1_c(S^1;\Db)$. Then for all $\vx \in \Db$, we have the following formulas:
\begin{enumerate}
    \item[(i)] $\displaystyle \vf\ (\vx) =  D_\vu \Lc_0 \vf(\vx)\,\vu(\vx)  +  D_\vu \Tc_0 \vf\ (\vx)\,\vu^\perp(\vx)$.
    \item[(ii)] Consider a conservative vector field $\displaystyle \vf =  \nabla \varphi$ (or a solenoidal vector field $\vf=\nabla^\perp \varphi$), 
    $\varphi \in C_c^2(\Db)$. The potential function $\varphi$ can be recovered from 
    $\Lc_0 \vf$ (respectively, from $\Tc_0 \vf$).
\end{enumerate}
\end{thr}
\begin{proof}
     Differentiating $\Lc_0\vf$ and $\Tc_0\vf$ in the direction $\vu$ yields
    \begin{align}
        D_{\vu}\Lc_0\vf &=\vu\cdot \vf,\label{du_Df}\\
        D_{\vu}\Tc_0\vf &=\vu^\perp \cdot \vf.\label{du_D_perpf}
    \end{align}
   The first statement then follows from a decomposition of $\vf$ along vectors $\vu$ and $\vu^\perp$.
   
    To prove the second statement, let us assume $\vf = \nabla \varphi$. Then 
    \begin{align}
        \Lc_{0}\vf (\vx) =- \int_0^\infty \vu(\vx) \cdot \nabla \varphi (\vx+t\vu(\vx)) \,dt = - \int_0^\infty \frac{d}{dt}\varphi (\vx+t\vu(\vx)) \,dt = \varphi (\vx).
    \end{align}
    The case $\vf = \nabla^\perp \varphi$ follows similarly using $\Tc_0 \vf$ instead of $\Lc_{0}\vf$.
\end{proof} 
We will see later that for $\alpha\neq0$, the longitudinal/transverse V-line transform and its first integral moment determine $\vf$ uniquely. As opposed to this, for $\alpha=0$, the integral moments do not give any new information as discussed in the lemma below. 
\begin{lem}Let $\vf\in C^1_c(S^1;\Db)$ and $k\ge 0$ be an integer. Then for all $\vx \in \Db$, we have the following relations: 
\begin{enumerate}
    \item[(i)] The integral data $\Lc_0^{k} \vf(\vx)$ and $\Lc_0^{k+1}\vf(\vx)$ are equivalent. 
    \item[(ii)] The integral data $\Tc_0^{k} \vf(\vx)$ and $\Tc_0^{k+1}\vf(\vx)$ are equivalent. 
\end{enumerate} 
\end{lem}
\begin{proof}
One can recover $\Lc_0^{k}\vf$ from $\Lc_0^{k+1}\vf$ using the formula  $D_{\vu}\Lc_0^{k+1}\vf = - (k+1)\Lc_0^{k}\vf.$
To obtain the statement in the opposite direction, let us integrate the above equation along $\vu$. 
\begin{align*}
      &\int_0^\infty \left(u_1(\vx+t\vu(\vx))\frac{\partial}{\partial x_1}+u_2(\vx+t\vu(\vx))\frac{\partial}{\partial x_2}\right) \Lc_0^{k+1}\vf(\vx+t\vu(\vx))\,dt\\ 
       = &\int_0^\infty \left(u_1(\vx)\frac{\partial}{\partial x_1}+u_2(\vx)\frac{\partial}{\partial x_2}\right) \Lc_0^{k+1}\vf(\vx+t\vu(\vx))\,dt = \int_0^\infty \frac{d}{dt} \Lc_0^{k+1}\vf(\vx+t\vu(\vx))\,dt\\
       = &- (k+1)\int_0^\infty \Lc_0^{k}\vf(\vx+t\vu(\vx))\,dt
  \end{align*}
  Since  $\Lc_0^{k+1}\vf(\vx+t\vu(\vx))=0$ for large $t$, we have $\Lc_0^{k+1}\vf(\vx)= (k+1)\int_0^\infty \Lc_0^{k}\vf(\vx+t\vu(\vx))\,dt .$  So, $\Lc_0^{k+1}\vf$ does not give us any new information. Therefore, $\Lc_0^{k+1}\vf$ and $\Lc_0^{k}\vf$ are equivalent. The proof of equivalence of $\Tc_0^{k}\vf$ and $\Tc_0^{k+1}\vf$ is similar. 
\end{proof}


\section{V-line transform of scalar functions (\texorpdfstring{$\alpha \neq 0$}{})}\label{sec: a not 0 scalar}
In this section, we analyze the weighted VLT of a scalar function 
and present a method to invert this transform. 
Some particular cases of this problem have been considered before. Katsevich-Krylov \cite{Kats_Krylov-13} studied the case when $\alpha = -1$, while $\vu$, $\vv$ are either constant or radial. They reduced the task of inverting the VLT to solving a first-order PDE with non-constant coefficients, which can be accomplished by the method of characteristics. Sherson \cite{Sherson} later derived explicit inversion formulas for the VLTs in the same setup. 
We solve the problem of inverting the weighted VLT for any $\alpha$ and arbitrary $\vu, \vv$ satisfying Hypothesis \ref{Hypo}. To achieve our result, we apply techniques similar to those used by Katsevich and Krylov \cite{Kats_Krylov-13}.
\begin{thr}\label{scalar_focal_rec}
    Let $h\in C_c^1(\Db).$ Then $h$ can be recovered from $\Vc_{\alpha}h$.
\end{thr}
\begin{proof}
Let us differentiate $\Vc_{\alpha}h(\vx)$ in the direction of $\vu$ and use the identities \eqref{decomp: D_u and D_v}, \eqref{D_u(u)=0} to get
    \begin{align}\label{D_uV_wh}
        D_\vu  \Vc_{\alpha}h(\vx)&=-h(\vx)+\alpha\left[c_{uv}(\vx)D_\vv+c^\perp_{uv}(\vx)D^\perp_\vv\right]\int_0^{\infty} h(\vx+t\vv(\vx))\, dt\\
        &=-\left[1+\alpha c_{uv}(\vx)\right]h(\vx) +\alpha c^\perp_{uv}(\vx)I_\vv(\vx), \quad \mbox{where } I_v(\vx)=D^\perp_\vv\int_0^{\infty} h(\vx+t\vv(\vx))\, dt.\nonumber
    \end{align}
Next, we apply $\displaystyle D_\vv$ to above relation and use identity \eqref{eq:D_v(c_uv)} to obtain
    \begin{align*}
        D_\vv D_\vu  \Vc_{\alpha}h(\vx) &= -\alpha D_\vv(c_{uv}(\vx))h(\vx)-\left[1 +\alpha c_{uv}(\vx)\right]D_\vv h(\vx)+\alpha D_\vv(c^\perp_{uv}(\vx))I_v(\vx)+\alpha c^\perp_{uv}(\vx)D_\vv I_v(\vx)\\
        &= -\alpha \delta \vu  (\vx)(c^\perp_{vu}(\vx))^2h(\vx)-\left[1 +\alpha c_{uv}(\vx)\right]D_\vv h(\vx)+\alpha \delta \vu  (\vx)c^\perp_{vu}(\vx)c_{uv}(\vx)I_v(\vx)\nonumber\\& \quad +\alpha c^\perp_{uv}(\vx)D_\vv I_v(\vx).
    \end{align*}
    Using identity \eqref{eq:D_uD^perp_u and D_vD^perp_v}, we find that $D_{\vv}I_v(\vx)= -D^\perp_{\vv}h(\vx)-\delta \vv(\vx) I_v(\vx)$. Substituting this into the above equation and using relations \eqref{decomp: D_u and D_v} and \eqref{D_uV_wh}, we obtain
     \begin{align}\label{scalar PDE}
         &(\alpha D_\vu + D_\vv)h(\vx) + \left[\alpha \delta \vu  (\vx)+ \delta \vv(\vx) + c_{uv}(\vx)\delta \vu  (\vx)+\alpha c_{uv}(\vx)\delta \vv(\vx)\right]h(\vx)\nonumber\\
         &=-[c_{uv}(\vx)\delta \vu  (\vx)+\delta \vv(\vx)]D_\vu\Vc_{\alpha}h(\vx) - D_\vv D_\vu \Vc_{\alpha}h(\vx).
     \end{align}
     This is a transport equation with respect to the unknown function $h$, which can be solved by the method of characteristics.
\end{proof}
\begin{rem}
The results of Theorem \ref{scalar_focal_rec} match with the statements in several previous works addressing particular cases of our setup. Namely,
\begin{itemize}
   \item If the vector fields $\vu$, $\vv$ are radial and $\alpha=-1$, our inversion of 
   $\Vc_{\alpha}$ 
  coincides with the inversion of the signed V-line transform derived in \cite{Kats_Krylov-13}. 
    \item  If the vector fields $\vu$, $\vv$ are constant (thus, $\delta \vu  $ and $\delta \vv$ are identically zero) and the weight $\alpha$ is arbitrary,
    our result reduces to the inversion formula obtained in \cite{amb-lat_2019}.
\end{itemize}
\end{rem}

\section{Longitudinal and transverse V-line transforms with their first moments (\texorpdfstring{$\alpha = 1$}{a1})}\label{sec: a = 1}
In this section we focus on the reconstruction of vector fields from their longitudinal and transverse VLTs, as well as from one of those transforms and its first moment. The reconstruction of $\vf$ using its longitudinal and transverse VLTs is discussed in Subsection \ref{subsec:L1f and T1f}, while the recovery from the longitudinal/transverse VLT and its first moment is addressed in Subsection \ref{subsec:L11f and T11f}. 

\subsection{Reconstruction of a vector field \texorpdfstring{$\vf$}{f} from \texorpdfstring{$\Lc_{1}\vf$}{Lf} and \texorpdfstring{$\Tc_{1}\vf$}{Tf}}\label{subsec:L1f and T1f}
First, we state the main results along with some corollaries and then present their proofs.

\begin{thr}[Kernel description]\label{th:kernel of L and T V-line}
    Let $\vf\in C^2_c(S^1;\Db)$. Then, we have 
    \begin{enumerate}
        \item[(i)] $\Lc_1\vf=0$ if and only if $\vf$ is conservative, i.e. $\displaystyle \vf = \nabla \varphi$, for some $\varphi \in C_c^3(\Db)$.
        \item[(ii)] $\Tc_1\vf=0$ if and only if $\vf$ is solenoidal, i.e.  $\displaystyle \vf = \nabla^\perp \varphi$, for some $\varphi \in C_c^3(\Db)$.
    \end{enumerate}
\end{thr}
\vspace{2mm}
\begin{thr}[Reconstruction formulas]\label{th:Inversion of L and T V-line}
    Let $\vf\in C^2_c(S^1;\Db)$. The following statements hold for all $\vx \in \Db$.
    \begin{enumerate}
        \item[(i)]  $\delta^\perp \vf$ can be explicitly recovered from $\Lc_1\vf$ as follows:
        \begin{align}\label{Lf:sol_rec}
            \delta^\perp \vf\ (\vx) = \frac{1}{\det (\vv,\vu)}\left[D_\vv D_\vu + \left\{\delta \vu(\vx)\,  c_{uv}(\vx) + \delta \vv(\vx)\right\}D_\vu\right]\Lc_1\vf (\vx).
        \end{align}
        \item[(ii)] $\delta \vf$ can be explicitly recovered from $\Tc_1\vf$ as follows:
        \begin{align}\label{Tf:pot_rec}
            \delta \vf\ (\vx) = -\frac{1}{\det (\vv,\vu)}\left[D_\vv D_\vu + \left\{\delta \vu(\vx)\, c_{uv}(\vx) + \delta \vv(\vx)\right\}D_\vu\right]\Tc_1\vf (\vx).
        \end{align}
    \end{enumerate}
\end{thr}


Note, the denominator $\displaystyle \det(\vv,\vu)$
is non-zero since $\vu$ and $\vv$  are linearly independent for all $\vx$.
\begin{cor}\label{rem:componentwise}
One can express the Laplacian of the components of $\vf$ in terms of $\delta^\perp \vf$ and $\delta \vf$ as follows: 
    \begin{align}
        \Delta f_1 = \frac{\partial}{\partial x_1} \delta \vf -  \frac{\partial}{\partial x_2} \delta^\perp \vf  \quad  \mbox{ and } \quad \Delta  f_2 = \frac{\partial}{\partial x_2} \delta \vf +  \frac{\partial}{\partial x_1} \delta^\perp \vf.
    \end{align}
    Hence, by Theorem \ref{th:Inversion of L and T V-line}, $\Delta f_1$ and $\Delta f_2$ can be written in terms of $\Lc_1 \vf$ and $\Tc_1 \vf$. 
    Therefore,  
    $\vf$ can be recovered from $\Lc_1 \vf$ and $\Tc_1 \vf$ by solving two Poisson equations with homogeneous boundary conditions. 
\end{cor}

\noindent Corollary \ref{rem:componentwise} provides an explicit componentwise reconstruction of the vector field $\vf = (f_1, f_2)$ from $\Lc_1 \vf$ and $\Tc_1 \vf$. Alternatively, one can recover the vector field  $\vf$  from the knowledge of $\Lc_1 \vf$ and $\Tc_1 \vf$ via decomposition \eqref{eq:decomposition}, as discussed below.
\begin{cor}
 Recall, from equation \eqref{eq:decomposition}, that any vector field $\vf$ can be uniquely decomposed as $$\vf= \nabla\varphi+ \nabla^{\perp}\psi, \quad \varphi|_{\partial \Db}=0, \quad \psi|_{\partial \Db}=0.$$
 Then the scalar functions $\psi$ and $\varphi$ can be recovered simultaneously from $\Lc_1 \vf$ and $\Tc_1 \vf$ respectively. 
\end{cor}
\begin{proof}
    By applying $\delta^\perp$ and $\delta$ to the decomposition of $\vf$,  we get, respectively,
    \begin{align*}
       \Delta \psi=  \delta^\perp\vf \quad \mbox{and} \quad  \Delta \varphi=  \delta\vf.
    \end{align*}
   Since $\delta^\perp\vf$ and $\delta \vf$ can be recovered from $\Lc_1\vf$ and  $\Tc_1\vf$ using equations \eqref{Lf:sol_rec}, \eqref{Tf:pot_rec}, one can recover $\varphi$ and $\psi$ by solving the corresponding Poisson equations with homogeneous boundary conditions.
\end{proof}

\begin{rem}
In the case when the vector fields $\vu$ and $\vv$ are constant (independent of position $\vx$), Theorems \ref{th:kernel of L and T V-line} and \ref{th:Inversion of L and T V-line} reduce to already established results \cite[Theorem 1, 2, 3, and 4]{Gaik_Mohammad_Rohit}. This can be seen by observing that $\delta \vu$ and $\delta \vv$ are zero when $\vu$ and $\vv$ are constant.
\end{rem}

\subsubsection{Proof of Theorem \ref{th:kernel of L and T V-line}}
\begin{proof}[\textbf{Proof of part (i)}]
    If $\vf= \nabla\varphi$, then by a direct calculation, we get $\Lc_1\vf=0.$ 
    The statement in the other direction follows from the reconstruction formula \eqref{Lf:sol_rec}. If $\Lc_1\vf=0,$ then by \eqref{Lf:sol_rec} we have $\delta^\perp \vf=0.$  For a simply connected domain, it is known that  $\delta^\perp \vf=0$ if and only if $\vf= \nabla\varphi$ for some scalar function $\varphi.$ This completes the proof.
\end{proof}
\begin{proof}[\textbf{Proof of part (ii)}]
    If $\vf= \nabla^\perp\psi$, then by a direct calculation, we get $\Tc_1\vf=0.$ 
    The statement in the other direction follows from the reconstruction formula \eqref{Tf:pot_rec}. If $\Tc_1\vf=0,$ then by \eqref{Tf:pot_rec} we have $\delta \vf=0.$ It is known that for a two-dimensional solenoidal vector field $\vf$ in a simply connected domain, there exists a scalar function $\psi$ such that $\vf= \nabla^\perp \psi$. This completes the proof.
\end{proof}


\subsubsection{Proof of Theorem \ref{th:Inversion of L and T V-line}}
\begin{proof}[\textbf{Proof of part (i)}]
Differentiating $\Lc_1\vf$ in the direction of $\vu$ and using relations \eqref{decomp: D_u and D_v}, \eqref{D_u(u)=0}, we get
\begin{align}\label{eq:D_uLf}
    D_\vu \Lc_1\vf = \vu \cdot \vf -c_{uv} \vv \cdot \vf + c^\perp_{uv} J_v, 
\end{align}
where  \begin{align}\label{eq:Jv}
    J_v(\vx)= D^\perp_\vv \int_0^\infty \vv(\vx) \cdot \vf(\vx+t\vv(\vx))\, dt.
    \end{align}
Next, applying the directional derivative along $\vv,$ we obtain
\begin{align}
   D_\vv D_\vu \Lc_1\vf = D_\vv(\vu \cdot \vf -c_{uv} \vv \cdot \vf) + D_\vv(c^\perp_{uv}) J_v + c^\perp_{uv} D_\vv J_v.
\end{align}
Using the relations  \eqref{eq:D_uD^perp_u and D_vD^perp_v} and \eqref{eq:Jv}, we get
\begin{align}
    D_\vv J_v(\vx)&=( D^\perp_\vv D_\vv - \delta \vv  D^\perp_\vv) \int_0^\infty \vv(\vx) \cdot \vf(\vx+t\vv(\vx))\, dt \nonumber \\
    \Longrightarrow \qquad  D_\vv J_v &= - D^\perp_\vv(\vv \cdot \vf)- \delta \vv J_v.\label{eq:DvJv}  
\end{align}
Using \eqref{eq:D_v(c_uv)} and \eqref{eq:DvJv}, we have
\begin{align*}
    D_\vv D_\vu \Lc_1\vf = D_\vv(\vu \cdot \vf - c_{uv} \vv \cdot \vf) -\delta \vu c^\perp_{uv}c_{uv} J_v - c^\perp_{uv} D^\perp_\vv(\vv \cdot \vf)- c^\perp_{uv}\delta \vv J_v .
\end{align*}
Then using \eqref{eq:D_uLf} and \eqref{decomp: D_u and D_v}, we get
\begin{align*}
    D_\vv D_\vu \Lc_1\vf &= D_\vv(\vu \cdot \vf - c_{uv} \vv \cdot \vf)-(D_\vu- c_{uv}D_\vv)(\vv \cdot \vf) - (\delta \vu  c_{uv}+\delta \vv)\left(D_{\vu} \Lc_1\vf-\vu \cdot \vf+c_{uv}\vv \cdot \vf\right)\\
    &=D_\vv(\vu\cdot \vf)-D_\vv(c_{uv})(\vv \cdot \vf)-D_\vu(\vv \cdot \vf)-(\delta \vu  c_{uv}+\delta \vv)(D_\vu \Lc_1\vf-\vu \cdot \vf+c_{uv}\vv \cdot \vf).
\end{align*}
Simplifying the expression using the aforementioned relations, we obtain
\begin{align*}
       & D_\vv D_\vu \Lc_1\vf + (\delta \vu  c_{uv}+\delta \vv)D_\vu \Lc_1\vf\\
     &= D_\vv(\vu)\cdot \vf + \vu \cdot D_\vv \vf - D_\vu(\vv)\cdot \vf - \vv \cdot  D_\vu \vf - \delta \vu  (c^\perp_{vu})^2 (\vv\cdot \vf) - (\delta \vu  c_{uv}+\delta\vv) (-\vu \cdot \vf+c_{uv}\vv \cdot \vf)\\ 
     &= c^\perp_{vu}\delta \vu  (\vu^\perp\cdot \vf)+ \vu \cdot D_\vv \vf - c^\perp_{uv}\delta \vv(\vv^\perp \cdot \vf)-\vv \cdot D_\vu \vf-  \delta \vu  (c^\perp_{vu})^2 (\vv\cdot \vf) + \delta \vu c_{uv}(\vu \cdot \vf)\\
     & \quad + \delta \vv(\vu \cdot \vf)- \delta \vu  c_{uv}^2(\vv \cdot \vf)- \delta \vv c_{uv}(\vv \cdot \vf), \quad \mbox{using relations \eqref {eq:D_u(v)} and \eqref{eq:D_v(u)}}
     \end{align*}
     \begin{align*}
     &= \vu \cdot D_\vv \vf - \vv \cdot D_\vu \vf + \delta \vu  [c^\perp_{vu}(\vu^\perp \cdot \vf)+ c_{uv}(\vu \cdot \vf)-\vv \cdot \vf] - \delta \vv[c^\perp_{uv}(\vv^\perp \cdot \vf)+c_{uv}(\vv \cdot \vf)-\vu \cdot \vf]\\
     &= \vu \cdot D_\vv \vf - \vv \cdot D_\vu \vf + \delta \vu  [(c^\perp_{vu}\vu^\perp + c_{uv}\vu) \cdot \vf-\vv \cdot \vf] - \delta \vv[(c^\perp_{uv}\vv^\perp +c_{uv} \vv ) \cdot \vf -\vu \cdot \vf ]\\
     &= \vu \cdot D_\vv \vf- \vv \cdot D_\vu \vf, \quad (\mbox{Since} ~ \vu= c_{uv}\vv + c_{uv}^{\perp}\vv^\perp ~ \mbox{and} ~ \vv= c_{vu}\vu + c_{vu}^{\perp}\vu^\perp)\\
     &= u_1v_1 \frac{\partial f_1}{\partial x_1}+u_1v_2\frac{\partial f_1}{\partial x_2}+u_2v_1\frac{\partial f_2}{\partial x_1} + u_2v_2 \frac{\partial f_2}{\partial x_2} -v_1u_1\frac{\partial f_1}{\partial x_1}- v_1u_2\frac{\partial f_1}{\partial x_2}-v_2u_1 \frac{\partial f_2}{\partial x_1}-v_2u_2\frac{\partial f_2}{\partial x_2}\\
     &= \det(\vv,\vu)\left(\frac{\partial f_2}{\partial x_1}-\frac{\partial f_1}{\partial x_2}\right).
\end{align*}
Hence, we have 
\begin{align}
    \delta^\perp \vf =  \frac{1}{\det(\vv,\vu)}[D_\vv D_\vu +(\delta \vu  c_{uv}+\delta \vv)D_\vu]\Lc_1\vf.
\end{align}
This completes the proof of part $(i)$ of Theorem \ref{th:Inversion of L and T V-line}.
\end{proof} 
\begin{proof}[\textbf{Proof of part (ii)}]
Let us apply the directional derivative $D_{\vu}$ to $\Tc_1\vf.$ Using again the identities \eqref{decomp: D_u and D_v} and \eqref{D_u(u)=0}, we get
\begin{align}\label{eq:D_uTf}
    D_\vu \Tc_1\vf = \vu^\perp \cdot \vf -c_{uv} \vv^\perp\cdot \vf 
    +c^\perp_{uv} J^\perp_v,
\end{align}
where  \begin{align}\label{eq:J^perp_v}
    J^\perp_v(\vx)= D^\perp_\vv \int_0^\infty \vv^\perp(\vx) \cdot \vf(\vx+t\vv(\vx))\, dt.
    \end{align}
Next applying directional derivative along $\vv,$ we obtain
\begin{align*}
   D_\vv D_\vu \Tc_1\vf = D_\vv(\vu^\perp \cdot \vf - c_{uv} \vv^\perp \cdot \vf) + D_\vv(c^\perp_{uv}) J^\perp_v + c^\perp_{uv} D_\vv J^\perp_v.
\end{align*}
Then using the identities \eqref{eq:D_uD^perp_u and D_vD^perp_v} and \eqref{eq:D_v(c_uv)}, we get
\begin{align*}
   D_\vv D_\vu \Tc_1\vf = D_\vv(\vu^\perp \cdot \vf) - D_\vv(c_{uv}) (\vv^\perp \cdot \vf) - c_{uv}D_\vv(\vv^\perp \cdot \vf) - \delta \vu c^\perp_{uv}c_{uv}J^\perp_v - c^\perp_{uv}D^\perp_\vv(\vv^\perp \cdot \vf) - c^\perp_{uv}\delta \vv J^\perp_v.
\end{align*}
Using the relations \eqref{decomp: D_u and D_v}, \eqref{eq:D_v(c_uv)} and \eqref{eq:D_uTf}, we obtain
\begin{align*}
   & D_\vv D_\vu \Tc_1\vf + (\delta\vu c_{uv}+\delta \vv)D_\vu \Tc_1\vf\\
    &= D_\vv(\vu^\perp \cdot \vf)- D_\vu(\vv^\perp \cdot \vf) - \delta \vu  (c^\perp_{vu})^2(\vv^\perp \cdot \vf) -(\delta \vu  c_{uv}+\delta \vv)(-\vu^\perp \cdot \vf + c_{uv}\vv^\perp \cdot \vf)\\
    &= D_\vv(\vu^\perp)\cdot \vf + \vu^\perp \cdot D_\vv \vf -D_\vu(\vv^\perp)\cdot \vf - \vv^\perp \cdot D_\vu \vf +(\delta \vu  c_{uv}+\delta \vv)(\vu^\perp \cdot \vf)-(\delta \vu  + \delta \vv c_{uv})(\vv^\perp \cdot \vf) \\
    &= \vu^\perp \cdot D_\vv \vf - \vv^\perp \cdot D_\vu \vf +\delta \vu  [c^\perp_{uv}\vu \cdot \vf + c_{uv}\vu^\perp \cdot \vf-\vv^\perp \cdot \vf]- \delta \vv[-c^\perp_{uv}\vv \cdot \vf + c_{uv}\vv^\perp \cdot \vf - \vu^\perp \cdot \vf]\\
    &= \vu^\perp \cdot D_\vv \vf- \vv^\perp \cdot D_\vu \vf \\
    &= - \det(\vv,\vu)\left( \frac{\partial f_1}{\partial x_1}+ \frac{\partial f_2}{\partial x_2}\right).
\end{align*}
Thus we have,
\begin{align}
    \delta \vf = -\frac{1}{\det(\vv,\vu)}[D_\vv D_\vu + (\delta \vu  c_{uv}+\delta \vv)D_\vu]\Tc_1\vf.
\end{align}
This completes the proof of part $(ii)$ of Theorem \ref{th:Inversion of L and T V-line}.
\end{proof} 
\subsection{Recovery of a vector field \texorpdfstring{$\vf$}{f} from \texorpdfstring{$\Lc_{1}\vf$}{Lf}/\texorpdfstring{$\Tc_{1}\vf$}{Lf} and \texorpdfstring{$\Lc^1_{1}\vf$}{Lf}/\texorpdfstring{$\Tc^1_{1}\vf$}{Tf}}\label{subsec:L11f and T11f}
In this subsection, we show that a vector field $\vf$ can be recovered in two different ways: from the combination of the longitudinal VLT $\Lc_{1}\vf$ and its first moment $\Lc^1_{1}\vf$, as well as from the knowledge of the transverse VLT $\Tc_{1}\vf$ and its first moment $\Tc^1_{1}\vf$. 
\begin{thr}\label{th:Inversion of L and L1 V-line}
    Any vector field $\vf \in C_c^2(S^1; \Db)$ 
    can be explicitly recovered from $\Lc_1\vf$ and $\Lc_1^1\vf.$
\end{thr}
\begin{proof}
Recall from equation \eqref{eq:decomposition} that a vector field $\vf$ can be decomposed as follows:
$$    \vf= \nabla\varphi+ \nabla^{\perp}\psi, \quad \varphi|_{\partial \Db}=0, \quad \psi|_{\partial \Db}=0. $$
We showed in the previous subsection that the scalar function $\psi$ is completely determined by the knowledge of $\Lc_1\vf$. To complete the proof of the theorem, we need to show that $\varphi$ can be recovered from the knowledge of reconstructed $\psi$ and $\Lc_1^1\vf$.
\vspace{2mm}\\
Applying $\Lc^1$ to the decomposition mentioned above, we get
    \begin{align*}
        \Lc_1^1\vf &= \Lc_1^1(\nabla\varphi)+\Lc_1^1(\nabla^\perp \psi)\\
        \Longrightarrow\qquad \qquad  \Lc_1^1\vf-\Lc^1_1(\nabla^\perp \psi)&= \Lc_1^1(\nabla\varphi)\\
        &= -\int_0^\infty t \vu(\vx)\cdot \nabla\varphi\left(\vx+ t \vu(\vx)\right)dt+\int_0^\infty t \vv(\vx)\cdot \nabla\varphi(\vx+ t \vv(\vx))dt\\
        &= -\int_0^\infty t \frac{d}{dt}\varphi(\vx+ t \vu(\vx))dt+\int_0^\infty t \frac{d}{dt}\varphi(\vx+ t \vv(\vx))dt\\
        & =  \int_0^\infty \varphi(\vx+ t \vu(\vx))dt-\int_0^\infty \varphi(\vx+ t \vv(\vx))dt
        =\Vc_{-1} \varphi (\vx).
    \end{align*}
    Using the inversion of $\Vc_{-1}$ discussed in Theorem  \ref{scalar_focal_rec}, we recover $\varphi$, which completes the proof.
\end{proof}

\begin{thr}\label{th:Inversion of T and T1 V-line}
    Let $\vf \in C_c^2(S^1; \Db).$ Then $\vf$ can be recovered explicitly from $\Tc_1\vf$ and $\Tc_1^1\vf.$
\end{thr}
\begin{proof}
We again start with the decomposition
$$    \vf= \nabla\varphi+ \nabla^{\perp}\psi, \quad \varphi|_{\partial \Db}=0, \quad \psi|_{\partial \Db}=0. $$
In this case, $\varphi$ is known from the knowledge of $\Tc_1\vf$, and we aim to recover $\psi$ using the additional information $\Tc_1^1\vf$.
Consider, 
\begin{align*}
\Tc_1^1\vf &= \Tc_1^1(\nabla\varphi)+\Tc_1^1(\nabla^\perp \psi)\\
        \Longrightarrow \quad \qquad  \Tc_1^1\vf-\Tc_1^1(\nabla\varphi) &= \Tc_1^1(\nabla^\perp \psi)\\
        &= -\int_0^\infty t \vu^\perp(\vx)\cdot \nabla^\perp\psi(\vx+ t \vu(\vx))dt+\int_0^\infty t \vv^\perp(\vx)\cdot \nabla^\perp\psi(\vx+ t \vv(\vx))dt\\
        &= -\int_0^\infty t \frac{d}{dt}\psi(\vx+ t \vu(\vx))dt+\int_0^\infty t \frac{d}{dt}\psi(\vx+ t \vv(\vx))dt\\
        & =  \int_0^\infty \psi(\vx+ t \vu(\vx))dt-\int_0^\infty \psi(\vx+ t \vv(\vx))dt
        =\Vc_{-1} \psi (\vx).
    \end{align*}
Using the inversion of $\Vc_{-1}$ discussed in Theorem  \ref{scalar_focal_rec}, we recover $\psi$, which completes the proof.
\end{proof}
\begin{rem}
The possibility of reconstructing a vector field $\vf$ from $\Lc_1\vf$ and $\Lc_1^1\vf$ (or $\Tc_1\vf$ and $\Tc_1^1\vf$) was shown before in \cite{Gaik_Mohammad_Rohit}, for the case when the branch directions $\vu$ and $\vv$ are constant. 
The method developed there recovered $f_1$ and $f_2$, while here we recover the parts of $\vf$ that come from decomposition \eqref{eq:decomposition}. Furthermore, the reconstruction method in \cite{Gaik_Mohammad_Rohit} was explicit, whereas here we need to solve certain partial differential equations. This disparity is due to the fact that in this paper we deal with a more general setup.
  
\end{rem}

\section{Longitudinal/transverse V-line transforms, when \texorpdfstring{$\alpha \ne0,1$}{} and \texorpdfstring{$\vu,\vv$}{u,v} are constant vector fields}\label{sec: a not 0 vector}
Throughout this section, we assume $\vu$ and $\vv$ are constant vector fields, and $\alpha \ne0,1$ is arbitrary. Without loss of generality, we can take V-lines symmetric about the $y$-axis, i.e. $\vu=(u_1,u_2)$ and $\vv=(-u_1,u_2)$. 

The idea here is to introduce new variables (depending on the weight $\alpha$) so that in the new coordinates, all the (weighted) integral transforms reduce to the unweighted case ($\alpha =1$). 
For $\alpha \ne0$, let us introduce the change of coordinates from $\vx=(x_1,x_2)$ to $\widetilde{\vx}=(\widetilde{x}_1, \widetilde{x}_2)$ as follows:
\begin{align}\label{change from x to tilde(x)}
    \widetilde{x}_1 = -\frac{1}{4\alpha}\left( \frac{1+\alpha}{u_1u_2}\,x_1+ \frac{1-\alpha}{u_2^2}\,x_2\right) \quad \mbox{and} \quad
    \widetilde{x}_2 = -\frac{1}{4\alpha}\left( \frac{1-\alpha}{u_1^2}\,x_1 +  \frac{1+\alpha}{u_1u_2}\, x_2\right).
\end{align}
The change of coordinates from $\tilde{\vx}=(\widetilde{x}_1,\widetilde{x}_2)$ to $\vx=(x_1,x_2)$ is given by 
\begin{align}\label{change from tilde(x) to x}
x_1 = -(1+\alpha)u_1u_2\, \widetilde{x}_1 +(1-\alpha)u_1^2\, \widetilde{x}_2 \quad \mbox{and} \quad
x_2 = (1-\alpha) u_2^2\, \widetilde{x}_1 -(1+\alpha)u_1u_2\, \widetilde{x}_2.  
\end{align}
Then, one may obtain by direct computation
\begin{align}\label{eq:derivatives in new coordinates}
   \partial_{ \widetilde{x}_1} = -(1+\alpha)u_1u_2\,\partial_{x_1}+(1-\alpha)u_2^2\,\partial_{x_2} \quad \mbox{ and }\quad \partial_{ \widetilde{x}_2} = (1-\alpha)u_1^2\,\partial_{x_1} - (1+\alpha)u_1u_2\,\partial_{x_2}.
\end{align}
The differential operators, such as gradient, divergence, and curl, are defined naturally in new coordinates as follows: 
\begin{align}\label{eq:div and curl in new coordinates}
\widetilde{\nabla} h = \left(\partial_{ \widetilde{x}_1} h,  \partial_{\widetilde{x}_2} h\right), \ \  \widetilde{\nabla}^\perp h = \left(-\partial_{\widetilde{x}_2} h, \partial_{\widetilde{x}_1} h\right), \ \    \widetilde{\delta} \vf =  \partial_{ \widetilde{x}_1}f_1 + \partial_{ \widetilde{x}_2}f_2,\ \  
\widetilde{\delta}^\perp \vf =  \partial_{ \widetilde{x}_1}f_2 - \partial_{ \widetilde{x}_2}f_1.
\end{align}

\subsection{Full recovery of a vector field \texorpdfstring{$\vf$}{} from \texorpdfstring{$\Lc_{\alpha}\vf$}{} and \texorpdfstring{$\Tc_{\alpha}\vf$}{}}
This subsection is dedicated to finding the kernels of $\Lc_{\alpha}\vf$ and $\Tc_{\alpha}\vf$, as well as recovering $\vf$ from these transforms. We show that each transform has a non-trivial null space and reconstruct the unknown vector field $\vf$ using the knowledge of both transforms $\Lc_{\alpha}\vf$ and $\Tc_{\alpha}\vf$.
\begin{lem} \label{lemma2}
     Let $\vf \in C_c^2(S^1;\Db).$ Then we have 
     \begin{align}
         D_{\vu}D_{\vv}\Lc_{\alpha}\vf &= \partial_{ \widetilde{x}_1}f_2 - \partial_{ \widetilde{x}_2}f_1 = \widetilde{\delta}^\perp \vf\label{DuDvL_alpf}. \\
        D_{\vu}D_{\vv}\Tc_{\alpha}\vf &= -(\partial_{ \widetilde{x}_1}f_1 + \partial_{ \widetilde{x}_2}f_2) = \widetilde{\delta} \vf \label{DuDvT_alpf}.
     \end{align}
\end{lem}
\begin{proof}
    Taking the directional derivatives of $\Lc_{\alpha}\vf$ in the directions of $\vu$ and $\vv$, we have
    \begin{align*}
        D_{\vu}D_{\vv} \Lc_{\alpha}\vf &= D_{\vv}(\vu\cdot \vf) -\alpha D_{\vu}(\vv\cdot \vf)\\
        &= (-u_1\partial_{x_1}+u_2\partial_{x_2})(u_1f_1+u_2f_2) -\alpha (u_1\partial_{x_1}+u_2\partial_{x_2})(-u_1f_1+u_2f_2)\\
        &= [-(1-\alpha)u_1^2\partial_{x_1}+(1+\alpha)u_1u_2\partial_{x_2}]f_1 +[-(1+\alpha)u_1u_2\partial_{x_1}+(1-\alpha)u_2^2\partial_{x_2}]f_2\\
        &=  \partial_{ \widetilde{x}_1}f_2 - \partial_{ \widetilde{x}_2}f_1 .
    \end{align*}
Similarly, taking the directional derivatives of $\Tc_{\alpha}\vf$ in the directions of $\vu$ and $\vv$, we have
\begin{align*}
     D_{\vu}D_{\vv} \Lc_{\alpha}\vf &= D_{\vv}(\vu^\perp\cdot \vf) -\alpha D_{\vu}(\vv^\perp\cdot \vf)\\
     &= (-u_1\partial_{x_1}+u_2\partial_{x_2})(-u_2f_1+u_1f_2)-\alpha(u_1\partial_{x_1}+u_2\partial_{x_2})(-u_2f_1-u_1f_2)\\
     &= [(1+\alpha)u_1u_2\partial_{x_1}-(1-\alpha)u_2^2]f_1 + [-(1-\alpha)u_1^2\partial_{x_1}+(1+\alpha)u_1u_2\partial_{x_2} ]f_2\\
     &=-(\partial_{ \widetilde{x}_1}f_1 + \partial_{ \widetilde{x}_2}f_2). \qedhere
\end{align*}   
\end{proof}

\begin{thr}[Kernel Description]\label{th:kernel of L_alpf}
    Let $\vf\in C^2_c(S^1;\Db)$. Then, we have 
    \begin{enumerate}
        \item[(i)] $\Lc_\alpha\vf=0$ if and only if $\displaystyle \vf =  \widetilde{\nabla} \varphi$ for some function $\varphi$.
        \item[(ii)] $\Tc_\alpha\vf=0$ if and only if $\displaystyle \vf =  \widetilde{\nabla}^\perp \psi$ for some function $\psi$.
    \end{enumerate}
\end{thr}
\begin{proof}[Proof of part (i)]
    Observe that $\Lc_{\alpha}\vf =0$ if and only if $ D_{\vu}D_{\vv}\Lc_{\alpha}\vf =0.$ Recall from Lemma \ref{lemma2} that $D_{\vu}D_{\vv}\Lc_{\alpha}\vf = \tilde{\delta}^\perp \vf$. It is known that in a simply connected domain $\tilde{\delta}^\perp \vf=0$ if and only if $\vf=\widetilde{\nabla}\varphi$ for some scalar function $\varphi$. This completes the proof.  
\end{proof}
\begin{proof}[Proof of part (ii)]
    Again, it is straightforward to observe that $\Tc_{\alpha}\vf =0$ if and only if $ D_{\vu}D_{\vv}\Tc_{\alpha}\vf =0$. From Lemma \ref{lemma2}  we know that  $D_{\vu}D_{\vv}\Tc_{\alpha}\vf = - \tilde{\delta} \vf$. Also, in simply connected domains,  $\tilde{\delta} \vf=0$ if and only if $\vf=\widetilde{\nabla}^\perp\psi$ for some scalar function $\psi,$ which finishes the proof.
\end{proof}
\begin{thr}\label{th:f from La and Ta}
Let $\vf \in C_c^2(S^1;\Db)$. Then $\vf$ can be recovered from the knowledge of $\Lc_{\alpha}\vf$ and $\Tc_{\alpha}\vf$.
\end{thr}
\begin{proof}
    From equations \eqref{DuDvL_alpf} and \eqref{DuDvT_alpf}, we have
    \begin{align*}
        D_{\vu}D_{\vv}\Lc_{\alpha}\vf =  \tilde{\delta}^\perp \vf & \quad \mbox{and } \quad D_{\vu}D_{\vv}\Tc_{\alpha}\vf = - \tilde{\delta} \vf.
    \end{align*}
    The Laplace operator in new coordinates is denoted by $\displaystyle \tilde{\Delta} := \frac{\partial^2}{\partial \widetilde{x}_1^2}+\frac{\partial^2}{\partial \widetilde{x}_2^2}.$ Then, the Laplacian of each component of $\vf$ can be expressed by the following relations
    \begin{align}
         \tilde{\Delta} f_1 = \frac{\partial}{\partial \widetilde{x}_1}\tilde{\delta} \vf - \frac{\partial}{\partial \widetilde{x}_2}\tilde{\delta}^\perp \vf, \label{eq: Laplace of f1 in new coordinates}\\
        \tilde{ \Delta } f_2 = \frac{\partial}{\partial \widetilde{x}_2}\tilde{\delta} \vf + \frac{\partial}{\partial \widetilde{x}_1}\tilde{\delta}^\perp \vf.\label{eq: Laplace of f2 in new coordinates}
    \end{align}
    Using these along with zero boundary conditions, we can uniquely recover $f_1$, $f_2$, and hence $\vf$.
\end{proof}
The proof of Theorem \ref{th:f from La and Ta} provides an algorithm for recovering $\vf$ by solving Poisson equations \eqref{eq: Laplace of f1 in new coordinates} and \eqref{eq: Laplace of f2 in new coordinates} for $f_1$ and $f_2$. If $\alpha=1$, this coincides with the corresponding result of \cite{Gaik_Mohammad_Rohit}.
\subsection{Full recovery of \texorpdfstring{$\vf$}{} using  integral moments}
In this subsection, we reconstruct $\vf$ using either the combinations of  $\Lc_{\alpha}\vf, \Lc^1_{\alpha}\vf$ or $\Tc_{\alpha}\vf,\Tc^1_{\alpha}\vf$. 
\begin{thr}\label{Vector_rec from La and La1}
    Let $\vf \in C_c^2(S^1; \Db).$ Then $\vf$ can be recovered from $\Lc_{\alpha}\vf$ and $\Lc_{\alpha}^1\vf$ using explicit closed form formulas \eqref{eq:rec f1 from La and La1} and \eqref{eq:rec f2 from La and La1}  (see below).
\end{thr}
\begin{proof}
Let us first note that under Hypothesis \ref{Hypo}, $\Lc_{\alpha}\vf$ can be expressed as follows:
\begin{align*}
\Lc_{\alpha}\vf= -\Xc_{\vu}\left(\vu\cdot\vf \right)+ \alpha \Xc_{\vv}\left(\vv\cdot\vf \right), \ \ \mbox{ where } \ \ \Xc_\vu h (\vx) := \int_0^\infty h(\vx + t\vu)dt
\end{align*}
   Since we have taken V-lines symmetric about the $y$-axis, that is, $\vu=(u_1,u_2)$ and $\vv=(-u_1,u_2)$, $\Lc_{\alpha}\vf$ can be further simplified as follows:
   \begin{align}\label{longi_alpha}
        \Lc_{\alpha}\vf= -u_1\Vc_{\alpha}f_1 - u_2 \Vc_{-\alpha}f_2.
    \end{align}
  It is known from \cite[Theorem 8]{amb-lat_2019}, that both $\Vc_\alpha$ and $\Vc_{-\alpha}$ can be inverted with explicit inversion formulas: 
   \begin{align}
       f_1 = \frac{1}{\lVert \vw_{\alpha}\rVert} D_{\vu}D_{\vv} \Xc_{\vw} \Vc_{\alpha}f_1 \quad \mbox{ and } \quad 
       f_2 = - \frac{1}{\lVert \widetilde{\vw}_{\alpha}\rVert} D_{\vu}D_{\vv} \Xc_{\widetilde{\vw}} \Vc_{-\alpha}f_2, 
   \end{align}
   where \begin{align*}
       \vw= \frac{\vw_{\alpha}}{\lVert \vw_{\alpha}\rVert},\ \ \vw_{\alpha}= (-(1-\alpha)u_1, (1+\alpha)u_2) \ \  \mbox{ and }\ \ 
       \widetilde{\vw}= \frac{\widetilde{\vw}_{\alpha}}{\lVert \widetilde{\vw}_{\alpha}\rVert},\ \  \widetilde{\vw}_{\alpha}= ((1+\alpha)u_1, -(1-\alpha)u_2).
   \end{align*}
    Applying $D_{\vu} D_{\vv}$ to $\Lc^1_{\alpha}\vf$ and using integration by parts, we get
    \begin{align*}
        D_{\vu}D_{\vv}\Lc^1_{\alpha}\vf = D_{\vv} \Xc_{\vu}(\vu \cdot \vf) - \alpha  D_{\vu} \Xc_{\vv}(\vv \cdot \vf).
    \end{align*}
    Next, using the identity $D_{\vu} + D_{\vv} = 2u_2\partial_{x_2}$, we have the following relation:
    \begin{align}\label{Relation L and L1}
     D_{\vu}D_{\vv}\Lc^1_{\alpha}\vf + (D_{\vu} + D_{\vv}) \Lc_{\alpha}\vf = (1+\alpha)u_1f_1 +(1-\alpha) u_2f_2.
    \end{align}    
    Using relation \eqref{longi_alpha} and inversion formulas for $\Vc_{\alpha},\Vc_{-\alpha}$, equation \eqref{Relation L and L1} can be rewritten as
    \begin{align}\label{D_uD_vL^1f}
    D_{\vu}D_{\vv}\Lc^1_{\alpha}\vf &= u_1\left\{ \frac{1+\alpha}{\lVert\vw_{\alpha}\rVert}  D_{\vu}D_{\vv} \Xc_{\vw} + (D_{\vu}
    + D_{\vv}) \right\} \Vc_{\alpha}f_1 \nonumber \\ & \qquad - u_2\left\{ \frac{1-\alpha}{\lVert\widetilde{\vw}_{\alpha}\rVert}  D_{\vu}D_{\vv} \Xc_{\widetilde{\vw}} - (D_{\vu} + D_{\vv}) \right\} \Vc_{-\alpha}f_2.
    \end{align}
    Applying the operator $ \displaystyle \left\{ \frac{1-\alpha}{\lVert\widetilde{\vw}_{\alpha}\rVert}  D_{\vu}D_{\vv} \Xc_{\widetilde{\vw}} - (D_{\vu} + D_{\vv}) \right\}$ to the equation \eqref{longi_alpha} and subtracting it from the equation \eqref{D_uD_vL^1f}, we obtain
    \begin{align*}
       D_{\vu}D_{\vv}\Lc^1_{\alpha}\vf - &\left\{ \frac{1-\alpha}{\lVert\widetilde{\vw}_{\alpha}\rVert}  D_{\vu}D_{\vv} \Xc_{\widetilde{\vw}} - (D_{\vu} + D_{\vv}) \right\} \Lc_{\alpha}\vf \nonumber \\ &\qquad = u_1\left\{ \frac{1+\alpha}{\lVert \vw_{\alpha}\rVert}D_{\vu}D_{\vv} \Xc_{\vw}+ \frac{1-\alpha}{\lVert \widetilde{\vw}_{\alpha}\rVert}D_{\vu}D_{\vv} \Xc_{\widetilde{\vw}} \right\}\Vc_{\alpha}f_1\\
       &\qquad = -u_1 \lVert \vw_{\alpha}\rVert \left\{ \frac{1+\alpha}{\lVert \vw_{\alpha}\rVert} \Xc_{\vw}+ \frac{1-\alpha}{\lVert \widetilde{\vw}_{\alpha}\rVert}\Xc_{\widetilde{\vw}} \right\}D_{\vw}f_1, 
    \end{align*}
    where in the last line we used the relation $D_{\vu}D_{\vv}\Vc_{\alpha}f_1= - D_{\vv}f_1 - \alpha D_{\vu}f_1= - \lVert \vw_{\alpha}\rVert D_{\vw} f_1$.
\vspace{2mm}\\
Denote by $\displaystyle C_{\vw\alpha}=\frac{1+\alpha}{\lVert\vw_{\alpha}\rVert}$, $\displaystyle C_{\widetilde{\vw}\alpha}= \frac{1-\alpha}{\lVert\widetilde{\vw}_{\alpha}\rVert}$ and $\displaystyle \gamma= \frac{C_{\vw\alpha}\widetilde{\vw}+ C_{\widetilde{\vw}\alpha}\vw}{\lVert C_{\vw\alpha}\widetilde{\vw}+ C_{\widetilde{\vw}\alpha}\vw\rVert}$. With these notations and using the inversion of the weighted V-line transform, we get
    \begin{align}
        D_{\vw}f_1= -\frac{1}{u_1}\frac{1}{\lVert \vw_{\alpha}\rVert}\frac{1}{\lVert C_{\vw\alpha}\widetilde{\vw}+ C_{\widetilde{\vw}\alpha}\vw\rVert}D_{\vw}D_{\widetilde{\vw}}\Xc_{\gamma}\left[ D_{\vu}D_{\vv}\Lc^1_{\alpha}\vf - \left\{ \frac{1-\alpha}{\lVert\widetilde{\vw}_{\alpha}\rVert}  D_{\vu}D_{\vv} \Xc_{\widetilde{\vw}} - (D_{\vu} + D_{\vv}) \right\} \Lc_{\alpha}\vf\right],
    \end{align}
    Integrating the equation along $\vw$, we have
    \begin{align}\label{eq:rec f1 from La and La1}
        f_1= -\frac{1}{u_1}\frac{1}{\lVert \vw_{\alpha}\rVert}\frac{1}{\lVert C_{\vw\alpha}\widetilde{\vw}+ C_{\widetilde{\vw}\alpha}\vw\rVert}D_{\widetilde{\vw}}\Xc_{\gamma}\left[ D_{\vu}D_{\vv}\Lc^1_{\alpha}\vf - \left\{ \frac{1-\alpha}{\lVert\widetilde{\vw}_{\alpha}\rVert}  D_{\vu}D_{\vv} \Xc_{\widetilde{\vw}} - (D_{\vu} + D_{\vv}) \right\} \Lc_{\alpha}\vf\right].
    \end{align}
Similarly applying the operator $\displaystyle \left\{ \frac{1+\alpha}{\lVert\vw_{\alpha}\rVert}  D_{\vu}D_{\vv} \Xc_{\vw} + (D_{\vu} + D_{\vv}) \right\}$ to equation \eqref{longi_alpha} and then adding it to equation \eqref{D_uD_vL^1f}, we get
\begin{align}
    D_{\vu}D_{\vv}\Lc^1_{\alpha}\vf &+ \left\{ \frac{1+\alpha}{\lVert\vw_{\alpha}\rVert}  D_{\vu}D_{\vv} \Xc_{\vw} + (D_{\vu} + D_{\vv}) \right\} \Lc_{\alpha}\vf \nonumber \\ &\qquad = -u_2 \lVert \widetilde{\vw}_{\alpha}\rVert \left\{ \frac{1+\alpha}{\lVert \vw_{\alpha}\rVert} \Xc_{\vw}+ \frac{1-\alpha}{\lVert \widetilde{\vw}_{\alpha}\rVert}\Xc_{\widetilde{\vw}} \right\}D_{\widetilde{\vw}}f_2.
\end{align}
Using the inversion of the weighted V-line transform, we get
    \begin{align}
        D_{\widetilde{\vw}}f_2= -\frac{1}{u_2}\frac{1}{\lVert \widetilde{\vw}_{\alpha}\rVert}\frac{1}{\lVert C_{\vw\alpha}\widetilde{\vw}+ C_{\widetilde{\vw}\alpha}\vw\rVert}D_{\vw}D_{\widetilde{\vw}}\Xc_{\gamma}\left[ D_{\vu}D_{\vv}\Lc^1_{\alpha}\vf + \left\{ \frac{1+\alpha}{\lVert\vw_{\alpha}\rVert}  D_{\vu}D_{\vv} \Xc_{\vw} + (D_{\vu} + D_{\vv}) \right\} \Lc_{\alpha}\vf\right].
    \end{align}
    Next, integrating the above equation in the direction $\widetilde{\vw},$ we obtain
    \begin{align}\label{eq:rec f2 from La and La1}
        f_2= -\frac{1}{u_2}\frac{1}{\lVert \widetilde{\vw}_{\alpha}\rVert}\frac{1}{\lVert C_{\vw \alpha}\widetilde{\vw}+ C_{\widetilde{\vw}\alpha}\vw\rVert}D_{\vw}\Xc_{\gamma}\left[ D_{\vu}D_{\vv}\Lc^1_{\alpha}\vf + \left\{ \frac{1+\alpha}{\lVert\vw_{\alpha}\rVert}  D_{\vu}D_{\vv} \Xc_{\vw} + (D_{\vu} + D_{\vv}) \right\} \Lc_{\alpha}\vf\right].
    \end{align}
Thus, $\vf$ can be recovered from the knowledge of $\Lc_{\alpha}\vf$ and $\Lc^1_{\alpha}\vf.$
\end{proof}
   
\begin{thr}\label{Vector_rec from Ta and Ta1}
    Let $\vf \in C_c^2(S^1; \Db).$ Then $\vf$ can be recovered from $\Tc_{\alpha}\vf$ and $\Tc_{\alpha}^1\vf$ using explicit closed form formulas \eqref{eq:rec f1 from Ta and Ta1} and \eqref{eq:rec f2 from Ta and Ta1} (given below).
\end{thr}
\begin{proof}
    Note that $\Tc_{\alpha}\vf = - \Lc_{\alpha}\vf^\perp$ and $\Tc^1_{\alpha}\vf = - \Lc^1_{\alpha}\vf^\perp.$ Using the similar procedure as above we can recover $\vf$ from $\Tc_{\alpha}\vf$ and $\Tc^1_{\alpha}\vf$ as given below
    \begin{align}
        f_1 &= \frac{1}{u_2}\frac{1}{\lVert \widetilde{\vw}_{\alpha}\rVert}\frac{1}{\lVert C_{\vw \alpha}\widetilde{\vw}+ C_{\widetilde{\vw}\alpha}\vw\rVert}D_{\vw}\Xc_{\gamma}\left[ D_{\vu}D_{\vv}\Tc^1_{\alpha}\vf + \left\{ \frac{1+\alpha}{\lVert\vw_{\alpha}\rVert}  D_{\vu}D_{\vv} \Xc_{\vw} + (D_{\vu} + D_{\vv}) \right\} \Tc_{\alpha}\vf\right]\label{eq:rec f1 from Ta and Ta1}\\
        f_2 &= -\frac{1}{u_1}\frac{1}{\lVert \vw_{\alpha}\rVert}\frac{1}{\lVert C_{\vw\alpha}\widetilde{\vw}+ C_{\widetilde{\vw}\alpha}\vw\rVert}D_{\widetilde{\vw}}\Xc_{\gamma}\left[ D_{\vu}D_{\vv}\Tc^1_{\alpha}\vf - \left\{ \frac{1-\alpha}{\lVert\widetilde{\vw}_{\alpha}\rVert}  D_{\vu}D_{\vv} \Xc_{\widetilde{\vw}} - (D_{\vu} + D_{\vv}) \right\} \Tc_{\alpha}\vf\right].\label{eq:rec f2 from Ta and Ta1}
    \end{align}
\end{proof}
\section{Open questions and future work}\label{sec:remarks}
\begin{enumerate}
    \item In this article, we studied a set of weighted V-line transforms acting on scalar functions and vector fields in $\Rb^2$, assuming that the integral curves of the branch vector fields of the V-lines are straight line segments inside the bounded image domain. 
    For transforms acting on vector fields, we presented inversion methods and kernel descriptions when the branch vector fields $\vu$ and $\vv$ are constant, or when the weight $\alpha$ is either 0 or 1. The reconstruction of a vector field with arbitrary $\alpha$ and non-constant vector fields $\vu$ and $\vv$, whose integral curves are straight line segments, is still an open problem. We hope that the techniques introduced here will be helpful in analyzing this general case as well and plan to address it in a future work. 
    \item The numerical implementation of inversion formulas presented here is an interesting and challenging task of its own. A set of algorithms dealing with previously examined particular cases were implemented in a recent work \cite{Gaik_Lat_Rohit_numerics}, and numerical reconstructions were stable and robust. We expect to obtain reconstructions of comparable quality using the methods developed in this paper and plan to address that in a future project. 
    \item One may extend the definition of the V-line transforms to higher-order tensor fields and ask similar questions about the kernel descriptions and invertibility. For special cases of $\alpha=1$ and constant vector fields $\vu$, $\vv$, this problem is considered in recent works \cite{Gaik_Indrani_Rohit} and \cite{Gaik_Indrani_Rohit_numerics}. One of our future goals is to extend the results of those works to weighted VLTs with non-constant vector fields $\vu$ and $\vv$, whose integral curves are straight line segments inside a bounded domain. 
\end{enumerate}
\section{Acknowledgements}\label{sec:acknowledge}
GA was partially supported by the NIH grant U01-EB029826. IZ was supported by the Prime Minister's Research Fellowship from the Government of India.
\bibliography{references}
\bibliographystyle{plain}

\end{document}